\documentclass[11pt]{article}
\usepackage{amsmath,amssymb,amsthm,mathtools}
\usepackage{enumitem}
\usepackage[margin=1in]{geometry}
\usepackage[hidelinks]{hyperref}

\newtheorem{theorem}{Theorem}[section]
\newtheorem{lemma}[theorem]{Lemma}
\newtheorem{corollary}[theorem]{Corollary}
\newtheorem{conjecture}[theorem]{Conjecture}
\newtheorem{observation}[theorem]{Observation}
\theoremstyle{remark}
\newtheorem{remark}[theorem]{Remark}
\theoremstyle{plain}
\numberwithin{equation}{section}

\newcommand{\ex}{\operatorname{ex}}
\newcommand{\cM}{\mathcal{M}}

\begin{document}
\title{On the Exact Tur\'an Number of $F^-_{4,3}$}
\author{ Chunqiu Fang \thanks{School of Computer Science and Technology, Dongguan University of Technology, Dongguan, Guangdong, 523808, China. Email: \texttt{fcq15@tsinghua.org.cn}. Supported by the National Natural Science Foundation for Young Scientists of China (Grant No.~12301435).}}
\date{}
\maketitle
\begin{abstract}
For a $3$-graph $F$, the Tur\'an number of $F$, denoted by $\ex(n,F)$, is the maximum number of edges in a $3$-graph on $n$ vertices containing no subgraph isomorphic to $F$.
Let $F^-_{4,3}$ be the $3$-graph formed by a complete four-vertex core and three outer vertices, with all but one of the twelve triples containing one core vertex and two outer vertices. We prove that, for every $n\ge8$,
\[
\ex(n,F^-_{4,3})=\binom n3-\binom{\lfloor n/2\rfloor}{3}-\binom{\lceil n/2\rceil}{3},
\]
and the balanced complete bipartite $3$-graph is the unique extremal configuration. This determines the exact value and all equality cases in the asymptotic theorem of Mubayi and R\"odl. It also extends the exact Tur\'an Number of $F_{3,3}$ and resolves a conjecture of Frankl, Huang and R\"odl.
\end{abstract}

\medskip\noindent\textbf{Keywords:} hypergraph Tur\'an number; $3$-uniform hypergraph; extremal configuration.

\section{Introduction}\label{sec:introduction}

An $r$-graph is a family of $r$-element subsets of a finite vertex set. For an $r$-graph $H$, write $V(H)$ for its vertex set, $E(H)$ for its edge set and $e(H)=|E(H)|$. For $X\subseteq V(H)$, let $H[X]$ be the subgraph induced by $X$ and $H-X=H[V(H)\setminus X]$. Write $K_s^{(r)}$ for the complete $r$-graph on $s$ vertices and call $X\subseteq V(H)$ complete if $H[X]\cong K_{|X|}^{(r)}$. A copy of an $r$-graph $F$ in $H$ is a subgraph of $H$ isomorphic to $F$. For a family $\mathcal F$ of $r$-graphs, we say that $H$ is $\mathcal F$-free if it contains no copy of any member of $\mathcal F$ and define the Tur\'an number of $\mathcal F$ by
\[
 \ex(n,\mathcal F)=\max\{e(H):H\text{ is an $\mathcal F$-free $r$-graph on $n$ vertices}\}.
\]
When $\mathcal F=\{F\}$, we write $\ex(n,F)$. A standard averaging argument shows that $\ex(n,\mathcal F)/\binom nr$ is nonincreasing for $n\ge r$. Consequently, the following limit exists:
\[
 \pi(\mathcal F)=\lim_{n\to\infty}\frac{\ex(n,\mathcal F)}{\binom nr}.
\]
It is called the Tur\'an density of $\mathcal F$; for a singleton family, we write $\pi(F)$.

Tur\'an's classical theorem determines both the extremal number and the unique extremal construction when the forbidden graph is complete~\cite{Turan}. For uniform hypergraphs, even the Tur\'an density of the tetrahedron $K_4^{(3)}$ remains unknown. Exact results that also characterize the extremal configurations are comparatively rare; see Keevash~\cite{KeevashSurvey} for a survey.

For positive integers $s$ and $t$, let $F_{s,t}$ be the $3$-graph with vertex partition $V_{1}\cup V_{2}$, where $|V_{1}|=s$ and $|V_{2}|=t$, and with edge set
\[
E(F_{s,t})=\binom{V_{1}}{3}\cup\{\{u,x,y\}:u\in V_{1},\ \{x,y\}\in\tbinom{V_{2}}{2}\}.
\]
We call $V_{1}$ the core and $V_{2}$ the outer set. Thus every edge of $F_{s,t}$ meets the core in either one or three vertices. For $s=4$ and $t=3$, we take $V_{1}=\{1,2,3,4\}$ and $V_{2}=\{5,6,7\}$ and write $F^-_{4,3}=F_{4,3}\setminus\{156\}$, where $156$ abbreviates the triple $\{1,5,6\}$. The isomorphism type does not depend on which such edge is deleted. Thus $F^-_{4,3}$ contains all four core triples and eleven of the twelve triples that meet the core in exactly one vertex.

A $3$-graph is two-colorable if its vertices can be colored with two colors so that no edge is monochromatic. Let $B_n$ denote the balanced complete bipartite $3$-graph obtained by partitioning an $n$-vertex set into two parts whose sizes differ by at most one and taking all triples that meet both parts. Let
\[
b(n)=e(B_n)=\binom n3-\binom{\lfloor n/2\rfloor}{3}-\binom{\lceil n/2\rceil}{3}.
\]
Among all two-colorable $3$-graphs on $n$ vertices, $B_n$ has the maximum number $b(n)$ of edges and $b(n)=(3/4+o(1))\binom n3$.

The most directly comparable exact result concerns $F_{3,3}$. Keevash and Mubayi~\cite{KeevashMubayi} proved that $\ex(n,F_{3,3})=b(n)$ for every $n\ge6$. Goldwasser and Hansen~\cite{GoldwasserHansen} independently obtained the same exact value and classified all extremal configurations, showing in particular that $B_n$ is the unique extremal $3$-graph for $n\ge6$. The vertices $\{2,3,4,5,6,7\}$ contain a copy of $F_{3,3}$ in $F^-_{4,3}$, with core $\{2,3,4\}$. Thus forbidding $F^-_{4,3}$ is a strictly weaker condition than forbidding $F_{3,3}$ and the exact theorem for $F_{3,3}$ does not provide the upper bound required here.

Mubayi and R\"odl~\cite{MubayiRodl} proved that $\pi(F^-_{4,3})=3/4$. Their argument uses the link-multigraph method developed by de Caen and F\"uredi for the Fano plane~\cite{deCaenFuredi}: after fixing a complete four-set, one superimposes the four corresponding link graphs to obtain an integer-weighted graph. Weighted versions of Tur\'an's theorem were studied by Bondy and Tuza~\cite{BondyTuza} and by F\"uredi and K\"undgen~\cite{FurediKundgen}. 

The balanced complete bipartite construction also arises in the Tur\'an problem for the Fano plane. De Caen and F\"uredi~\cite{deCaenFuredi} first determined its Tur\'an density. The exact result for sufficiently large orders was obtained independently by Keevash and Sudakov~\cite{KeevashSudakov} and by F\"uredi and Simonovits~\cite{FurediSimonovits}; Bellmann and Reiher~\cite{BellmannReiher} later established the exact formula for every $n\ge7$. The $3$-graph $B_n$ is the unique extremal configuration for $n\ge8$, whereas at order seven there is one additional extremal configuration. Our main theorem determines the exact extremal number and the unique extremal configuration for $F^-_{4,3}$ from order eight onward.

\begin{theorem}\label{thm:main}
Let $H$ be an $F^-_{4,3}$-free $3$-graph on $n$ vertices, where $n\ge8$. Then $e(H)\le b(n)$, with equality if and only if $H$ is isomorphic to $B_n$.
\end{theorem}

Frankl, Huang and R\"odl~\cite[Conjecture~4.2]{FranklHuangRodl} considered the two $3$-graphs
\[
R_1=F_{4,3}\setminus\{156,257,367\},\qquad R_2=F_{4,3}\setminus\{156,157,367\}.
\]
Both $R_1$ and $R_2$ are subgraphs of $F^-_{4,3}$. Also, one can check that both $R_1$ and $R_2$ are not two-colorable. Thus, $B_{n}$ is $\{R_{1}, R_{2}\}$-free.
 They posed the following conjecture, stated here in our notation.

\begin{conjecture}[Frankl--Huang--R\"odl]\label{conj:FHR}
There is an integer $m_0$ such that, for every integer $m>m_0$,
\[
  \ex(2m,\{R_1,R_2\})=\binom{2m}{3}-2\binom m3.
\]
\end{conjecture}

Theorem~\ref{thm:main} implies the following result for every order at least eight, including the characterization of equality.

\begin{corollary}\label{cor:R}
For every $n\ge8$, $\ex(n,\{R_1,R_2\})=b(n)$. Moreover, up to isomorphism, $B_n$ is the unique extremal $3$-graph. In particular, Conjecture~\ref{conj:FHR} holds with $m_0=3$.
\end{corollary}

\begin{proof}
Every $\{R_1,R_2\}$-free $3$-graph is $F^-_{4,3}$-free because both $R_1$ and $R_2$ are subgraphs of $F^-_{4,3}$. Hence Theorem~\ref{thm:main} establishes the upper bound and its equality case. Conversely, $B_n$ contains neither $R_1$ nor $R_2$. Thus $B_n$ attains $b(n)$ edges and the final assertion follows by taking $n=2m$.
\end{proof}

\medskip\noindent\textit{Outline of the proofs.} Following the deletion approach of Keevash and Mubayi~\cite{KeevashMubayi}, we induct on the order in steps of four by deleting a complete four-set. A sharp bound on the combined weight of its four link graphs controls the edges that meet the four-set exactly once. Under the deletion hypothesis, a counterexample to the upper bound would contain a $K_5^{(3)}$, and then a six-vertex subgraph with at least nineteen edges. Local link estimates exclude both possible six-vertex configurations. At equality, every vertex outside a complete four-set has exactly one missing pair in its link on that four-set. These pairs fall into two complementary types, which determine the balanced bipartition of the whole graph.

\medskip\noindent\textit{Organisation.} In Section~\ref{sec:prelim}, we establishes the local estimates and the seven-vertex classification. In Section~\ref{sec:proof}, we proves the deletion lemma, handles the initial orders $8,9,10,11$ and completes the induction for the upper bound and uniqueness.

\section{Preliminaries and local lemmas}\label{sec:prelim}

For a $3$-graph $H$, a set $S\subseteq V(H)$ and $j\in\{1,2,3\}$, let
 $\mathcal E_j^H(S)=\{e\in E(H):|e\cap S|=j\}$.
We omit the superscript when the ambient graph is clear. The number of edges meeting $S$ is denoted by $\psi_H(S)=e(H)-e(H-S)=\sum_{j=1}^3|\mathcal E_j^H(S)|$.
For $v\in V(H)$ and $X\subseteq V(H)\setminus\{v\}$, denote the link graph of $v$ restricted to $X$ by $ L_H(v,X)=\{xy\in\tbinom X2:vxy\in E(H)\}$. We omit the subscript $H$ from $\psi_H$ and $L_H$ when the ambient graph is clear. For a pair $xy$ and a set $S$ disjoint from $\{x,y\}$, its codegree in $S$ is $|\{w\in S:wxy\in E(H)\}|$. For a graph $G$ and a vertex $x\in V(G)$, write $d_G(x)$ for the degree of $x$ in $G$. For $P\subseteq V(G)$, let $\partial_G(P)$ denote the set of edges of $G$ with exactly one endpoint in $P$.

Define
\[
  \gamma(m)=2\binom m2+2\left\lfloor\frac{m^2}{4}\right\rfloor
           =3\binom m2+\left\lfloor\frac m2\right\rfloor.
\]
Direct calculation shows that the following identities hold for integers $n\ge4$, $m\ge3$ and $t\ge1$:
\begin{align}
 b(n)-b(n-4)&=\gamma(n-4)+5(n-4)+4,                 \label{eq:b-inc}\\
 \gamma(m)-\gamma(m-2)&=6m-8,                       \label{eq:g-two}\\
 \gamma(t+1)-\gamma(t)&=
 \begin{cases}
  3t,&t\text{ even},\\
  3t+1,&t\text{ odd}.
 \end{cases}                                        \label{eq:g-one}
\end{align}

For every nonnegative integer $k$, we have the elementary identities
\[
 b(2k)-\frac34\binom{2k}{3}=\frac{k(k-1)}2,
 \qquad
 b(2k+1)-\frac34\binom{2k+1}{3}=\frac{k(2k-1)}4.
\]
In particular, $b(n)>\frac34\binom n3$ for $n\ge4$.

\begin{observation}\label{obs:four-set-incidence}
Every $n$-vertex $3$-graph $H$ with $n\ge4$ and more than $\frac34\binom n3$ edges contains a complete four-set.
\end{observation}

\begin{proof}
Suppose that $H$ contains no complete four-set. Count the pairs $(e,R)$ for which $e\in E(H)$ and $R$ is a four-set containing $e$. Each edge occurs in exactly $n-3$ such pairs, whereas every four-set contains at most three edges. Thus $(n-3)e(H)\le3\binom n4=\frac34(n-3)\binom n3$ and hence $e(H)\le\frac34\binom n3$, contrary to the hypothesis.
\end{proof}

If $S$ is a complete four-set in $H$ and $U\subseteq V(H)\setminus S$, define a weight on $\binom U2$ by
\[
  \omega_S(xy)=|\{u\in S:uxy\in E(H)\}|.
\]
Equivalently, $\omega_S(xy)$ is the multiplicity of $xy$ in the multiset union of the four link graphs restricted to $U$.

The next bound is an application of the weighted graph inequality of Bondy and Tuza~\cite{BondyTuza}; see also~\cite{FurediKundgen,KeevashMubayi}. We include a proof for completeness.

\begin{lemma}\label{lem:link}
Let $S$ be a complete four-set in an $F^-_{4,3}$-free $3$-graph $H$ and let $U\subseteq V(H)\setminus S$ have size $m$. Then $\sum_{xy\in\binom U2}\omega_S(xy)\le\gamma(m)$.
\end{lemma}

\begin{proof}
Fix $T\in\binom U3$. The sum $\sum_{xy\in\binom T2}\omega_S(xy)$ counts the twelve possible extensions $uxy$ with $u\in S$ and $xy\in\binom T2$. If this sum is at least eleven, then $H$ contains a copy of $F^-_{4,3}$ with core $S$ and outer set $T$, a contradiction. Therefore $\sum_{xy\in\binom T2}\omega_S(xy)\le10$.

We prove the bound by induction on $m$. The cases $m\le2$ are immediate. If $\omega_S(xy)\le3$ for every $xy\in\binom U2$, then the total weight is at most $3\binom m2\le\gamma(m)$. Otherwise, choose $xy$ with $\omega_S(xy)=4$. For every $z\in U\setminus\{x,y\}$, the preceding inequality implies that $\omega_S(xz)+\omega_S(yz)\le6$. Applying the induction hypothesis on $U\setminus\{x,y\}$ and using~\eqref{eq:g-two}, we obtain
\[
 \sum_{pq\in\binom U2}\omega_S(pq)
 \le\gamma(m-2)+4+6(m-2)=\gamma(m).
\]
\end{proof}

\begin{lemma}\label{lem:sparse-cut}
Every graph $G$ on six vertices with at most four edges contains a two-set $P\subseteq V(G)$ such that $|\partial_G(P)|\le1$.
\end{lemma}

\begin{proof}
If $G$ has an isolated vertex $x$, then its degree sum is at most
eight, so some other vertex $y$ has degree at most one and
$|\partial_G(\{x,y\})|\le1$. Otherwise $G$ is disconnected,
since a connected graph on six vertices has at least five edges.
If $G$ has a $K_2$ component, its vertices form the required pair.
If not, every component has at least three vertices, so
$G=P_3\cup P_3$. Two adjacent vertices in either component
have cut size one.
\end{proof}

\begin{lemma}\label{lem:K5-pair}
Let $H$ be an $F^-_{4,3}$-free $3$-graph and $X\subseteq V(H)$ span a $K_5^{(3)}$. Suppose that $u$ and $v$ are distinct vertices of $V(H)\setminus X$ with $|L(u,X)|=|L(v,X)|=8$. Then
\[
  |\{x\in X:xuv\in E(H)\}|\le3.
\]
\end{lemma}

\begin{proof}
Let $M_u$ and $M_v$ consist of the two pairs missing from $L(u,X)$ and $L(v,X)$, respectively, and let $M$ be their multiset union. Write
$C=\{x\in X:xuv\in E(H)\}$ and let $d(x)$ be the degree
of $x$ in $M$, counted with multiplicity.

Suppose that $|C|\ge4$. For each $z\in X$, the core $X\setminus\{z\}$ and
outer set $\{z,u,v\}$ have exactly
$d(z)+|(X\setminus C)\setminus\{z\}|$ missing extensions.
This number is at least two, since otherwise these vertices contain $F^-_{4,3}$. Since $\sum_{z\in X}d(z)=8$,
we cannot have $C=X$. Thus $X\setminus C=\{q\}$,
$d(q)\ge2$, and $d(x)\ge1$ for every $x\in C$.

Since $d(q)\ge2$ and $M$ has four edges, at most two lie inside $C$. If $d(x)\ge2$
for some $x\in C$, choose $y\in C\setminus\{x\}$ with
$xy\notin M$. Then $\{u,v,x,y\}$ is complete, and its
missing extensions to the outer set $X\setminus\{x,y\}$
are precisely the occurrences of $M$ disjoint from $\{x,y\}$.
Their number is $4-d(x)-d(y)\le1$, a contradiction.

Consequently, every vertex of $C$ has degree one and all four
edges of $M$ contain $q$. Relabel $C=\{x,y,z,p\}$ so that
$M_v=\{qx,qy\}$ and $M_u=\{qz,qp\}$.
The core $\{v,q,z,p\}$ is complete and all twelve extensions
to the outer set $\{x,y,u\}$ are present. This is again a
copy of $F^-_{4,3}$, completing the proof.
\end{proof}

The next lemma controls links and pair codegrees around the two six-vertex $3$-graphs with at least nineteen edges.

\begin{lemma}\label{lem:six-links}
Let $H$ be an $F^-_{4,3}$-free $3$-graph and let $X\in\binom{V(H)}6$.
\begin{enumerate}[label=\textup{(\roman*)},leftmargin=2.2em]
\item If $H[X]=K_6^{(3)}$, then $|L(u,X)|\le10$ for every $u\in V(H)\setminus X$.
\item Suppose $X=A\cup B$, $|A|=|B|=3$ and $H[X]=K_6^{(3)}\setminus\{A\}$. Then the following statements hold.
\begin{enumerate}[label=\textup{(\alph*)},leftmargin=2.2em]
\item We have $|L(u,X)|\le12$ for every $u\in V(H)\setminus X$. If equality holds for some $u$, then
\begin{equation}
  L(u,X)=\binom X2\setminus\binom A2. \label{eq:equality-link}
\end{equation}
We call the link in~\eqref{eq:equality-link} the equality-case link.
\item If $u$ and $v$ are distinct vertices outside $X$ that both have the equality-case link, then
\[
  |\{x\in X:xuv\in E(H)\}|\le3.
\]
\item If $u$ and $z$ are distinct vertices outside $X$ such that $u$ has the equality-case link and $xuz\in E(H)$ for every $x\in X$, then $|L(z,X)|\le10$.
\end{enumerate}
\end{enumerate}
\end{lemma}

\begin{proof}
For (i), fix $u\in V(H)\setminus X$ and let $M$ be the graph on $X$ with $E(M)=\binom X2\setminus L(u,X)$. Suppose that $e(M)\le4$. By Lemma~\ref{lem:sparse-cut}, there is a two-set $P\subseteq X$ such that $|\partial_M(P)|\le1$. Since $H[X]$ is complete, the four-set $X\setminus P$ is complete and the pair $P$ has four extensions in this four-set. Moreover, $|\partial_{L(u,X)}(P)|=8-|\partial_M(P)|\ge7$. Relative to the core $X\setminus P$ and the outer set $P\cup\{u\}$, the other two outer pairs therefore have at least seven extensions in total. Hence $H$ contains a copy of $F^-_{4,3}$ with this core and outer set, a contradiction. Thus $e(M)\ge5$ and consequently $|L(u,X)|=15-e(M)\le10$.

Assume now the setting of (ii). To prove (a), fix $u\in V(H)\setminus X$. For every mixed pair $P=\{x,y\}$, where $x\in A$ and $y\in B$, the four-set $X\setminus P$ is complete and $P$ has four extensions in this four-set. We must have $|\partial_{L(u,X)}(P)|\le6$; otherwise $H$ contains a copy of $F^-_{4,3}$ with core $X\setminus P$ and outer set $P\cup\{u\}$, a contradiction. Let $s$ and $t$ be the numbers of same-side and cross pairs in $L(u,X)$, respectively. Summing this inequality over the nine mixed pairs counts each same-side edge six times and each cross edge four times, so $6s+4t\le54$. Since $t\le9$, we have $6(s+t)\le54+2t\le72$ and therefore $|L(u,X)|=s+t\le12$.

Suppose that equality holds. Since $t\le9$, we have $54\ge6s+4t=72-2t\ge54$. Thus $t=9$, $s=3$ and every mixed-pair cut has size six. For $x\in A$ and $y\in B$, exactly four of the nine cross edges lie in the cut determined by $\{x,y\}$. Hence $d_{L(u,A)}(x)+d_{L(u,B)}(y)=2$. Comparing this identity for two choices of $x$ and for two choices of $y$ shows that the degrees are constant on each of $A$ and $B$. Since a simple graph on three vertices cannot be one-regular, one of $L(u,A)$ and $L(u,B)$ is a triangle and the other is empty. Suppose that the triangle lies in $A$. For any $x\in A$, relative to the core $\{x\}\cup B$ and the outer set $(A\setminus\{x\})\cup\{u\}$, the pair $A\setminus\{x\}$ has weight three and each pair containing $u$ has weight four. The total weight is eleven, so $H$ contains a copy of $F^-_{4,3}$ with core $\{x\}\cup B$ and outer set $(A\setminus\{x\})\cup\{u\}$, a contradiction. Thus $L(u,B)$ is the triangle and~\eqref{eq:equality-link} follows.

For (b), write $C=\{x\in X:xuv\in E(H)\}$. Let $S$ be
any four-set containing two vertices from each of $A$ and $B$,
and let $y$ be the vertex of $B\setminus S$. Relative to the
complete core $S$ and outer set $\{y,u,v\}$, the pairs $yu$
and $yv$ both have weight four. Hence $|C\cap S|\le2$.
Summing over the nine choices of $S$, each vertex occurs
six times, so $6|C|\le18$ and therefore $|C|\le3$.

For (c), fix $x\in A$ and let $Y=(X\setminus\{x\})\cup\{u\}$.
The induced graph $H[Y]$ is complete except for the triple
$\{u\}\cup(A\setminus\{x\})$. By part (a), $|L(z,Y)|\le12$.
Equality would require both pairs joining $u$ to
$A\setminus\{x\}$ to be missing from $L(z,Y)$, whereas
the hypothesis includes both. Thus $|L(z,Y)|\le11$.
All five pairs joining $u$ to $X\setminus\{x\}$ belong to
$L(z,Y)$, so $|L(z,X\setminus\{x\})|\le6$.
Summing over $x\in A$, every edge of $L(z,X)$ is counted
at least twice except the at most three edges inside $A$,
which are counted once. Therefore
\[
 2|L(z,X)|-3
 \le\sum_{x\in A}|L(z,X\setminus\{x\})|
 \le18.
\]
Since $|L(z,X)|$ is an integer, it is at most ten.
\end{proof}

The initial cases also require the following seven-vertex bound. We include the classification of the thirty-edge graphs under the additional $K_5^{(3)}$-free assumption.

\begin{lemma}\label{lem:seven}
Let $H$ be an $F^-_{4,3}$-free $3$-graph on seven vertices.
\begin{enumerate}[label=\textup{(\roman*)},leftmargin=2.2em]
\item We have $e(H)\le31$. Equality holds if and only if $H\cong K_7^{(3)}\setminus K_4^{(3)}$.

\item If $H$ is also $K_5^{(3)}$-free and $e(H)=30$, then $H\cong B_7$.
\end{enumerate}
\end{lemma}

\begin{proof}
Let $\cM=\binom{V(H)}3\setminus E(H)$ be the family of missing triples, and let $d(x)$ count its members containing $x$. A transversal of $\cM$ is a vertex set meeting every member. Every three-element transversal $T$ satisfies
\begin{equation}
 \bigl|\{M\in\cM:|M\cap T|=2\}\bigr|\ge2. \label{eq:transversal}
\end{equation}
Otherwise, $H$ contains $F^-_{4,3}$ with the complete core $V(H)\setminus T$ and outer set $T$.

For (i), suppose that $|\cM|\le4$. There is a transversal of size at most three: for four members, choose a common vertex of two intersecting members and one vertex from each of the others. A minimum transversal of size three has three distinct members of $\cM$ meeting it in one vertex each, leaving at most one member that meets it twice. This contradicts~\eqref{eq:transversal}. Thus $\cM$ has a two-element transversal $\{x,y\}$. Applying~\eqref{eq:transversal} to the five triples $\{x,y,z\}$, we obtain
\[
 10\le\sum_{z\notin\{x,y\}}
 \bigl|\{M\in\cM:|M\cap\{x,y,z\}|=2\}\bigr|
 =2(d(x)+d(y)).
\]
Indeed, a member containing just one of $x,y$ contributes twice, and a member containing both contributes four times. We may therefore assume that $d(x)\ge3$. If every member contains $x$, its two-element trace on $V(H)\setminus\{x\}$ forms an edge of a graph with at most four edges. By Lemma~\ref{lem:sparse-cut}, some pair $P$ has cut size at most one, contradicting~\eqref{eq:transversal} for $\{x\}\cup P$. Consequently, $|\cM|=4$ and $d(x)=3$.

Let $A$ be the unique member avoiding $x$ and let $B=V(H)\setminus(A\cup\{x\})$. The three members containing $x$ have traces forming a three-edge graph $G$ on $A\cup B$. For every $u\in A$ and $v\in B$, the triple $\{x,u,v\}$ is a transversal, so~\eqref{eq:transversal} implies $|\partial_G(\{u,v\})|\ge2$. Let $c$ count the edges of $G$ between $A$ and $B$. Summing these nine inequalities, we obtain
\[
 18\le\sum_{u\in A,\,v\in B}|\partial_G(\{u,v\})|
 =6(3-c)+4c=18-2c.
\]
Thus $c=0$ and every cut has size two. In particular, for any $u\in A$ and $ v\in B$, we have $d_{G[A]}(u)+d_{G[B]}(v)=2.$ Both induced graphs are regular. A regular simple graph on three vertices is either empty or a triangle, so exactly one of $G[A],G[B]$ is a triangle. If it is $G[B]$, then $x$ together with two vertices of $A$ forms a transversal meeting only $A$ twice, contrary to~\eqref{eq:transversal}. Hence $G[A]$ is the triangle and $\cM=\binom{A\cup\{x\}}3$.

This proves $e(H)\le31$ and the necessity of the equality condition. Conversely, $K_7^{(3)}\setminus K_4^{(3)}$ is $F^-_{4,3}$-free: any copy would be spanning and would have an independent four-set, whereas every four-set of $F^-_{4,3}$ contains an edge.

For (ii), $|\cM|=5$, and the $K_5^{(3)}$-free assumption means that $\cM$ has no two-element transversal. Hence $d(x)\le3$ for every vertex $x$: a vertex in at least four members, together with one vertex from the remaining member if necessary, would meet all five. Since $\sum_xd(x)=15>14$, some vertex $x$ has degree three. The two members $A,B$ avoiding $x$ must be disjoint, since otherwise $x$ and a common vertex would form a two-element transversal. Thus $V(H)\setminus\{x\}=A\cup B$.

Again let $G$ be the three-edge graph formed by the traces of the members containing $x$. For $u\in A$ and $v\in B$, both $A$ and $B$ meet $\{x,u,v\}$ once, so the same cut calculation shows that $G$ is a triangle on one part and empty on the other. Interchanging $A,B$ if necessary, we have
 $\cM=\binom{A\cup\{x\}}3\cup\{B\}$.
These are precisely the monochromatic triples of a $4+3$ bipartition, so $H\cong B_7$.
\end{proof}

\begin{remark}
The preceding lemma identifies the last exceptional order before the range of Theorem~\ref{thm:main}. For $n\le6$, we have $\ex(n,F^-_{4,3})=\binom n3$, whereas $\ex(7,F^-_{4,3})=31$ by Lemma~\ref{lem:seven}(i), with a unique extremal $3$-graph. Since $b(5)=9$, $b(6)=18$ and $b(7)=30$, the bipartite formula fails exactly at orders five, six and seven. Together with the cases $n\le6$, Theorem~\ref{thm:main} and Lemma~\ref{lem:seven}(i) therefore determine $\ex(n,F^-_{4,3})$ for every $n$.
\end{remark}

\section{Proof of the main theorem}\label{sec:proof}

The next lemma contains the inductive step for both the upper bound and the equality case.

\begin{lemma}\label{lem:engine}
Let $n\ge8$ and let $H$ be an $F^-_{4,3}$-free $3$-graph on $n$ vertices. Suppose that
\begin{equation}
  e(H-R)\le b(n-4)\qquad\text{for every complete four-set $R\subseteq V(H)$.} \label{eq:delete-assumption}
\end{equation}
Then $e(H)\le b(n)$, with equality if and only if $H\cong B_n$.

\end{lemma}

\begin{proof}
We first prove the upper bound and show that equality excludes a $K_5^{(3)}$. Suppose either that $e(H)\ge b(n)+1$, or that $e(H)=b(n)$ and $H$ contains a $K_5^{(3)}$. Let $\delta=1$ in the first case and $\delta=0$ in the second. We will derive a contradiction in both cases. By~\eqref{eq:delete-assumption} and~\eqref{eq:b-inc}, every complete four-set $R$ satisfies
\begin{equation}
  \psi(R)\ge\gamma(n-4)+5(n-4)+4+\delta. \label{eq:psi-core}
\end{equation}
When $\delta=1$, Observation~\ref{obs:four-set-incidence} supplies a complete four-set $R_0$. Lemma~\ref{lem:link} bounds $|\mathcal E_1(R_0)|$ by $\gamma(n-4)$, so~\eqref{eq:psi-core} implies $|\mathcal E_2(R_0)|\ge5(n-4)+1$. Hence some vertex outside $R_0$ is joined to all six pairs of $R_0$ and $R_0$ extends to a $K_5^{(3)}$. When $\delta=0$, such a $K_5^{(3)}$ is part of the assumption. In either case, let $X$ denote its vertex set.

Apply~\eqref{eq:psi-core} to the five complete four-sets $X\setminus\{x\}$, where $x\in X$, and sum the resulting inequalities. Summing Lemma~\ref{lem:link} over the same cores, with outer set $V(H)\setminus X$, produces the second estimate:
\begin{align}
 50+5|\mathcal E_2(X)|+4|\mathcal E_1(X)|
   &\ge5[\gamma(n-4)+5(n-5)+9+\delta],             \label{eq:K5-inc}\\
 4|\mathcal E_1(X)|&\le5\gamma(n-5).               \notag
\end{align}
Indeed, an edge inside $X$ or in $\mathcal E_2(X)$ is counted five times in the first sum, an edge in $\mathcal E_1(X)$ is counted four times in both sums and the ten internal edges contribute fifty to the first sum. Consequently,
\begin{equation}
 |\mathcal E_2(X)|\ge\gamma(n-4)-\gamma(n-5)+5(n-5)-1+\delta.
 \label{eq:E2-lower}
\end{equation}

We next show that some $u\in V(H)\setminus X$ satisfies $|L(u,X)|\ge9$. Suppose instead that every such link has at most eight edges. Equations~\eqref{eq:E2-lower} and~\eqref{eq:g-one} imply
\[
 8(n-5)-1\le|\mathcal E_2(X)|\le8(n-5).
\]
Thus either all external links have size eight, or exactly one has size seven and all others have size eight. In the former case, Lemma~\ref{lem:K5-pair} and~\eqref{eq:K5-inc} imply
\[
3\binom{n-5}{2}\ge|\mathcal E_1(X)|
\ge\frac{5\gamma(n-5)-5}{4}
>3\binom{n-5}{2},
\]
a contradiction. In the latter case,~\eqref{eq:E2-lower} and~\eqref{eq:g-one} force $\delta=0$ and $n-5$ to be even. Hence
\[
|\mathcal E_1(X)|
\le3\binom{n-6}{2}+5(n-6)
<\frac54\gamma(n-5)
\le|\mathcal E_1(X)|.
\]
The strict inequality holds because $n-5$ is even and at least four; the difference is $(n-7)(3n-23)/8>0$. Both cases are impossible. We may therefore choose $u$ with $|L(u,X)|\ge9$ and let $Y=X\cup\{u\}$; then $e(H[Y])\ge19$.

We now exclude both possibilities for $H[Y]$.
Write $m=n-6\ge2$. Suppose first that $H[Y]=K_6^{(3)}$. Sum $\psi(R)$ over the fifteen four-subsets $R$ of $Y$. Each internal edge is counted fifteen times, each edge in $\mathcal E_2(Y)$ fourteen times and each edge in $\mathcal E_1(Y)$ ten times. Lemma~\ref{lem:six-links}(i) implies $|\mathcal E_2(Y)|\le10m$. Summing Lemma~\ref{lem:link} over these cores, with outer vertices restricted to $V(H)\setminus Y$, also shows that $10|\mathcal E_1(Y)|\le15\gamma(m)$. Hence
\[
 \sum_{R\in\binom Y4}\psi(R)
 \le300+140m+15\gamma(m).
\]
On the other hand,~\eqref{eq:psi-core} and~\eqref{eq:g-two} imply
\[
 \sum_{R\in\binom Y4}\psi(R)
 \ge15[\gamma(m)+11m+18+\delta].
\]
These bounds force $25m\le30-15\delta$, contrary to $m\ge2$.

Thus $H[Y]=K_6^{(3)}\setminus\{A\}$. Write $Y=A\cup B$, where $|A|=|B|=3$ and consider the nine complete four-sets containing two vertices from each side. Assign weight nine to an edge $qxy$ with $q\notin Y$ and $x,y$ on the same side and weight eight when $x,y$ lie on different sides. Let $\omega$ denote the total weight. Sum~\eqref{eq:psi-core} over the nine cores and separately sum Lemma~\ref{lem:link} over their restrictions to $V(H)\setminus Y$. We obtain
\begin{align}
 9[\gamma(m)+11m+18+\delta]
   &\le171+\omega+6|\mathcal E_1(Y)|,               \label{eq:nine-inc}\\
 6|\mathcal E_1(Y)|&\le9\gamma(m).                \label{eq:nine-link}
\end{align}
Here the nineteen internal edges are counted nine times, a same-side pair is counted nine times, a mixed pair eight times and an edge in $\mathcal E_1(Y)$ six times. Combining the two inequalities, we obtain
\begin{equation}
 \omega\ge99m-9+9\delta.                       \label{eq:omega-lower}
\end{equation}

By Lemma~\ref{lem:six-links}(ii)(a), an external vertex contributes at most $99$ to $\omega$, with equality precisely for the equality-case link. Every other link has at most eleven edges and therefore contributes at most $6\cdot9+5\cdot8=94$. Let $r$ be the number of external vertices whose links are not of the equality type. Then $\omega\le99m-5r$, so~\eqref{eq:omega-lower} implies $5r\le9(1-\delta)$. Thus $r=0$ when $\delta=1$, and $r\le1$ when $\delta=0$. Moreover,~\eqref{eq:nine-inc} and this upper bound on $\omega$ imply
\begin{equation}
 |\mathcal E_1(Y)|\ge\frac32\gamma(m)-\frac32+\frac32\delta+\frac56r.
 \label{eq:E1-six-lower}
\end{equation}

If $r=0$, Lemma~\ref{lem:six-links}(ii)(b) implies $|\mathcal E_1(Y)|\le3\binom m2$, whereas~\eqref{eq:E1-six-lower} implies
\[
 |\mathcal E_1(Y)|\ge\frac32\gamma(m)-\frac32
 >3\binom m2,
\]
where the strict inequality holds for every $m\ge2$.

It remains to consider $r=1$, which forces $\delta=0$. Let $z$ be the exceptional vertex. The other $m-1$ vertices each contribute $99$ to $\omega$, so~\eqref{eq:omega-lower} forces the contribution of $z$ to be at least $90$. A link with at most ten edges contributes at most $6\cdot9+4\cdot8=86$. Hence $|L(z,Y)|=11$. By Lemma~\ref{lem:six-links}(ii)(c), every pair consisting of $z$ and an equality-case vertex has codegree at most five in $Y$. Lemma~\ref{lem:six-links}(ii)(b) bounds the codegree of every other external pair by three. Consequently,
\[
 |\mathcal E_1(Y)|\le3\binom{m-1}{2}+5(m-1).
\]
Since $\gamma(m)$ is even and $|\mathcal E_1(Y)|$ is an integer,~\eqref{eq:E1-six-lower} also implies $|\mathcal E_1(Y)|\ge\frac32\gamma(m)$. This is impossible, because
\[
 \frac32\gamma(m)-\left[3\binom{m-1}{2}+5(m-1)\right]
 =\frac34(m-1)(m-2)
  +\frac{3\lfloor m/2\rfloor-m+1}{2}>0.
\]
Here $m\ge2$ and $3\lfloor m/2\rfloor\ge m$. We conclude that $e(H)\le b(n)$, and that $H$ is $K_5^{(3)}$-free whenever equality holds.

It remains to characterize equality. Assume that $e(H)=b(n)$, so $H$ is $K_5^{(3)}$-free by the preceding argument. For every complete four-set $R$, each external vertex is joined to at most five pairs of $R$; otherwise it would form a $K_5^{(3)}$ with $R$. By Lemma~\ref{lem:link} and~\eqref{eq:delete-assumption},
\[
 b(n)-b(n-4)\le\psi(R)
 \le4+5(n-4)+\gamma(n-4)=b(n)-b(n-4).
\]
Equality holds throughout: $e(H-R)=b(n-4)$, $|\mathcal E_1(R)|=\gamma(n-4)$ and $|\mathcal E_2(R)|=5(n-4)$. Since each external link on $R$ has size at most five, every such link has size exactly five. Thus
\begin{equation}
\bigl|\{uv\in\tbinom R2:uvx\notin E(H)\}\bigr|=1
\quad\text{for every complete four-set $R$ and every $x\notin R$.} \label{eq:unique-missing-pair}
\end{equation}

By Observation~\ref{obs:four-set-incidence}, choose a complete four-set $S$ and write $U=V(H)\setminus S$. For $x\in U$, let $\mu(x)\in\binom S2$ be its unique missing pair. For distinct $x,y\in U$, let
\[
 D(x,y)=\{q\in S:qxy\notin E(H)\}.
\]
If $q\in\mu(x)$, then $(S\setminus\{q\})\cup\{x\}$ is a complete four-set. Among the pairs contained in $S\setminus\{q\}$, the link of $y$ is missing no pair if $q\in\mu(y)$, and exactly one otherwise. The other missing pairs are counted by $|D(x,y)\setminus\{q\}|$. Thus~\eqref{eq:unique-missing-pair} implies
\begin{equation}
 |D(x,y)\setminus\{q\}|=
 \begin{cases}
  1,&q\in\mu(y),\\
  0,&q\notin\mu(y).
 \end{cases} \label{eq:mu-identity}
\end{equation}
Suppose that $\mu(x)=\{u,v\}$ and $\mu(y)=\{u,w\}$, where $v\ne w$. Taking $q=v$ in~\eqref{eq:mu-identity} shows that $D(x,y)\subseteq\{v\}$. After interchanging $x$ and $y$, taking $q=w$ shows that $D(x,y)\subseteq\{w\}$. Thus $D(x,y)=\varnothing$. On the other hand, taking $q=u$ requires $|D(x,y)\setminus\{u\}|=1$, a contradiction. Hence two distinct values of $\mu$ are complementary. In this case, applying~\eqref{eq:mu-identity} to both vertices of $\mu(x)$ implies $D(x,y)=\varnothing$. If $\mu(x)=\mu(y)=P=\{u,v\}$, then $|D(x,y)\setminus\{u\}|=|D(x,y)\setminus\{v\}|=1$. Thus either $D(x,y)=P$, or $D(x,y)$ is a singleton in $S\setminus P$.

Fix $P$ in the image of $\mu$ and let $Q=S\setminus P$. By the preceding paragraph, every value of $\mu$ is either $P$ or $Q$. Moreover,
\[
 |\mathcal E_1(S)|=\sum_{xy\in\binom U2}(4-|D(x,y)|)=\gamma(n-4)>3\binom{n-4}{2}.
\]
The displayed inequality forces $D(x,y)=\varnothing$ for some pair $xy$. The alternatives in the preceding paragraph then rule out $\mu(x)=\mu(y)$. Hence $\mu(x)\ne\mu(y)$ and both types $P$ and $Q$ occur. If $p,q\in U$ have the same type and $z\in U$ has the opposite type, then $D(p,z)=D(q,z)=\varnothing$. If $|D(p,q)|\le1$, then $H$ contains a copy of $F^-_{4,3}$ with core $S$ and outer set $\{p,q,z\}$, a contradiction. Therefore $|D(p,q)|\ge2$ and the preceding alternatives imply
\begin{equation}
 D(x,y)=
 \begin{cases}
  P,&\mu(x)=\mu(y)=P,\\
  Q,&\mu(x)=\mu(y)=Q,\\
  \varnothing,&\mu(x)\ne\mu(y).
 \end{cases}                                                     \label{eq:D-structure}
\end{equation}
Let $U_P,U_Q$ be the two type classes.

We claim that each type class is independent. Suppose that $x,y,z\in U_P$ form an edge and choose $q\in Q$. By~\eqref{eq:D-structure}, the four-set $\{q,x,y,z\}$ is complete. However, for any $p\in P$, its link on this four-set is missing the three pairs $xy,xz,yz$, contrary to~\eqref{eq:unique-missing-pair}. Thus $U_P$ is independent, and the same argument applies to $U_Q$.

The definition of $\mu$ and~\eqref{eq:D-structure} now show that $P\cup U_P$ and $Q\cup U_Q$ are both independent. Thus $H$ is two-colorable. If these parts have sizes $a,b$, then
\[
 b(n)=e(H)\le\frac{ab(n-2)}2
 \le\frac{n-2}{2}\left\lfloor\frac{n^2}{4}\right\rfloor=b(n).
\]
Equality forces $|a-b|\le1$ and every triple meeting both parts to be present. Hence $H\cong B_n$.
\end{proof}

The induction based on Lemma~\ref{lem:engine} advances by four, so we first settle the orders $8,9,10,11$.

\begin{lemma}\label{lem:bases}
Let $H$ be an $F^-_{4,3}$-free $3$-graph on $n$ vertices, where $8\le n\le11$. Then $e(H)\le b(n)$, with equality if and only if $H\cong B_n$.
\end{lemma}

\begin{proof}
For $n=8$, deleting any four vertices leaves at most $\binom43=4=b(4)$ edges. Lemma~\ref{lem:engine} therefore establishes the bound and its equality case.

We next treat $n=9$ and $n=11$ together. It suffices to consider $e(H)\ge b(n)$. We shall exclude six-sets spanning at least nineteen edges, then exclude complete five-sets, and finally apply Lemma~\ref{lem:engine}. Write $t=n-6\in\{3,5\}$. We first show that every six-set spans at most eighteen edges. Suppose otherwise, and choose a six-set $Y$ spanning at least nineteen edges. Since its complement has $t$ vertices,
\begin{equation}
 \psi(Y)\ge b(n)-\binom t3=
 \begin{cases}
 69,&t=3,\\
 125,&t=5.
 \end{cases}                                                    \label{eq:base-six-lower}
\end{equation}
The induced graph $H[Y]$ is either $K_6^{(3)}$ or $K_6^{(3)}$ with one edge removed. In the first case, sum Lemma~\ref{lem:link} over all fifteen four-subsets of $Y$, restricting the outer vertices to $V(H)\setminus Y$. Each edge in $\mathcal E_1(Y)$ occurs ten times. In the second case, write $H[Y]=K_6^{(3)}\setminus\{A\}$ and use the nine complete four-sets with two vertices in each of $A$ and $Y\setminus A$; each edge in $\mathcal E_1(Y)$ occurs six times. Thus in either case,
\[
 |\mathcal E_1(Y)|\le\frac32\gamma(t)=
 \begin{cases}
 15,&t=3,\\
 48,&t=5.
 \end{cases}
\]
If $H[Y]$ is complete, Lemma~\ref{lem:six-links}(i) implies
\[
 \psi(Y)\le20+10t+\frac32\gamma(t)=
 \begin{cases}
 65,&t=3,\\
 118,&t=5,
 \end{cases}
\]
contrary to~\eqref{eq:base-six-lower}.

We may therefore assume that $H[Y]=K_6^{(3)}\setminus\{A\}$. Let $a$ be the number of vertices outside $Y$ whose links on $Y$ are the equality-case link from Lemma~\ref{lem:six-links}(ii)(a). The remaining $t-a$ links have at most eleven edges. Moreover, part (b) bounds the number of extensions of each pair among these $a$ vertices by three, while every other pair has at most six extensions in $Y$. Together with the preceding link bound, these observations imply $|\mathcal E_2(Y)|\le11t+a$ and $|\mathcal E_1(Y)|\le
 \min\left\{\frac32\gamma(t),6\binom t2-3\binom a2\right\}$.
Consequently,
\[
 \psi(Y)\le19+11t+a+
 \min\left\{\frac32\gamma(t),6\binom t2-3\binom a2\right\}.
\]
For $t=3$, the first term of the minimum bounds this expression by $67+a\le69$ when $a\le2$, while the second bounds it by $64$ when $a=3$. For $t=5$, the corresponding bounds are $122+a\le125$ when $a\le3$ and at most $120$ when $a\ge4$. Comparison with~\eqref{eq:base-six-lower} leaves only $(t,a)\in\{(3,2),(5,3)\}$. In either case, $|\mathcal E_2(Y)|=11t+a$ and $|\mathcal E_1(Y)|=\frac32\gamma(t)$. In particular, every other external link has size eleven.

A pair consisting of an equality-case vertex and a vertex with an eleven-edge link cannot have six extensions in $Y$: Lemma~\ref{lem:six-links}(ii)(c) would force the latter link to have at most ten edges. Thus each such pair has at most five extensions. Pairs of equality-case vertices have at most three extensions by part (b), while all remaining pairs have at most six. Therefore
\[
 |\mathcal E_1(Y)|\le
 3\binom a2+5a(t-a)+6\binom{t-a}{2}=
 \begin{cases}
 13,&(t,a)=(3,2),\\
 45,&(t,a)=(5,3).
 \end{cases}
\]
Both values are smaller than $\frac32\gamma(t)$, a contradiction. This proves the six-set estimate.

We now show that $H$ is $K_5^{(3)}$-free. Suppose that a five-set $X$ is complete, and write $s=n-5\in\{4,6\}$. The six-set estimate shows that every external link on $X$ has at most eight edges. It also implies $e(H-X)\le18=b(6)$ when $s=6$; for $s=4$, the bound $e(H-X)\le4=b(4)$ is immediate. Hence
\[
 \psi(X)\ge b(n)-b(s)=
 \begin{cases}
 66,&s=4,\\
 117,&s=6.
 \end{cases}
\]
Let $a$ be the number of external links on $X$ with eight edges. Then $|\mathcal E_2(X)|\le7s+a$. Summing Lemma~\ref{lem:link} over the five complete four-subsets of $X$, with outer vertices restricted to $V(H)\setminus X$, counts each edge in $\mathcal E_1(X)$ four times. This bounds $|\mathcal E_1(X)|$ by $\frac54\gamma(s)$. Alternatively, Lemma~\ref{lem:K5-pair} bounds the extensions into $X$ of each pair among the $a$ vertices by three; every other pair has at most five extensions. Therefore
\[
 |\mathcal E_1(X)|\le
 \min\left\{\frac54\gamma(s),5\binom s2-2\binom a2\right\}.
\]
Thus
\[
 \psi(X)\le10+7s+a+
 \min\left\{\frac54\gamma(s),5\binom s2-2\binom a2\right\}.
\]
For $s=4$, the right-hand side is at most $65$: use the first term of the minimum when $a\le2$ and the second when $a\ge3$. For $s=6$, it is at most $116$: use the first term when $a\le4$ and the second when $a\ge5$. Both bounds contradict the preceding lower bound on $\psi(X)$. Therefore $H$ is $K_5^{(3)}$-free.

For $n=9$, this implies $e(H-R)\le9=b(5)$ for every complete four-set $R$. For $n=11$, Lemma~\ref{lem:seven}(i) implies $e(H-R)\le30=b(7)$: its exceptional thirty-one-edge graph contains a $K_5^{(3)}$, formed by the three vertices outside the missing four-set and any two vertices inside it. Thus the deletion hypothesis of Lemma~\ref{lem:engine} holds in both cases. The lemma establishes $e(H)\le b(n)$ and $H\cong B_n$ at equality.

Finally, suppose that $n=10$. By the $n=9$ case, every vertex deletion has at most $70$ edges. Each edge occurs in exactly seven of these deletions, so
\[
 7e(H)=\sum_{v\in V(H)}e(H-v)\le10\cdot70.
\]
Hence $e(H)\le100=b(10)$. If equality holds, every $H-v$ is isomorphic to $B_9$, so $H$ is $K_5^{(3)}$-free. A six-vertex $3$-graph with at least nineteen edges contains a $K_5^{(3)}$: if one triple is missing, delete a vertex of that triple. Consequently, $e(H-R)\le18=b(6)$ for every complete four-set $R$. Lemma~\ref{lem:engine} now implies $H\cong B_{10}$.
\end{proof}

\begin{proof}[Proof of Theorem~\ref{thm:main}]
The construction $B_n$ supplies the lower bound. $B_n$ is $F^-_{4,3}$-free and $\ex(n,F^-_{4,3})\ge b(n)$. Lemma~\ref{lem:bases} establishes the matching upper bound and uniqueness for $8\le n\le11$. Proceed by induction in steps of four. Let $n\ge12$, assume that the theorem holds at order $n-4$ and let $H$ be an $F^-_{4,3}$-free $3$-graph on $n$ vertices. For every complete four-set $R$, the induced graph $H-R$ is $F^-_{4,3}$-free, so the induction hypothesis implies $e(H-R)\le b(n-4)$. Lemma~\ref{lem:engine} establishes the bound and its equality case, completing the induction.
\end{proof}

\section*{Declaration of AI use}

The authors used AI tools to assist with developing ideas, checking the proofs, and improving the language. All AI-assisted calculations and suggestions were independently verified and revised by the authors. The authors take full responsibility for the content of the paper.


{\small
\begin{thebibliography}{99}

\bibitem{BellmannReiher} L. Bellmann and C. Reiher, \emph{Tur\'an's theorem for the Fano plane}, Combinatorica \textbf{39} (2019), no.~5, 961--982.

\bibitem{BondyTuza} J. A. Bondy and Zs. Tuza, \emph{A weighted generalization of Tur\'an's theorem}, J. Graph Theory \textbf{25} (1997), no.~4, 267--275.

\bibitem{deCaenFuredi} D. de Caen and Z. F\"uredi, \emph{The maximum size of $3$-uniform hypergraphs not containing a Fano plane}, J. Combin. Theory Ser. B \textbf{78} (2000), no.~2, 274--276.

\bibitem{FranklHuangRodl} P. Frankl, H. Huang and V. R\"odl, \emph{On local Tur\'an problems}, J. Combin. Theory Ser. A \textbf{177} (2021), Paper No.~105329, 9~pp.

\bibitem{FurediKundgen} Z. F\"uredi and A. K\"undgen, \emph{Tur\'an problems for integer-weighted graphs}, J. Graph Theory \textbf{40} (2002), no.~4, 195--225.

\bibitem{FurediSimonovits} Z. F\"uredi and M. Simonovits, \emph{Triple systems not containing a Fano configuration}, Combin. Probab. Comput. \textbf{14} (2005), no.~4, 467--484.

\bibitem{GoldwasserHansen} J. Goldwasser and R. Hansen, \emph{The exact Tur\'an number of $F(3,3)$ and all extremal configurations}, SIAM J. Discrete Math. \textbf{27} (2013), no.~2, 910--917.

\bibitem{KeevashSurvey} P. Keevash, \emph{Hypergraph Tur\'an problems}, in: Surveys in Combinatorics 2011, London Math. Soc. Lecture Note Ser. \textbf{392}, Cambridge Univ. Press, Cambridge, 2011, 83--140.

\bibitem{KeevashMubayi} P. Keevash and D. Mubayi, \emph{The Tur\'an number of $F_{3,3}$}, Combin. Probab. Comput. \textbf{21} (2012), no.~3, 451--456.

\bibitem{KeevashSudakov} P. Keevash and B. Sudakov, \emph{The Tur\'an number of the Fano plane}, Combinatorica \textbf{25} (2005), no.~5, 561--574.

\bibitem{MubayiRodl} D. Mubayi and V. R\"odl, \emph{On the Tur\'an number of triple systems}, J. Combin. Theory Ser. A \textbf{100} (2002), no.~1, 136--152.

\bibitem{Turan} P. Tur\'an, \emph{On an extremal problem in graph theory} (in Hungarian), Mat. Fiz. Lapok \textbf{48} (1941), 436--452.

\end{thebibliography}
}
\end{document}